\documentclass[english]{article}

\usepackage{lmodern}
\usepackage[a4paper]{geometry}

\usepackage[utf8]{inputenc}
\usepackage[T1]{fontenc}
\usepackage{amsthm}
\usepackage{amsmath} 
\usepackage{amsfonts}
\usepackage{hyperref}
\hypersetup{colorlinks,allcolors=black}
\usepackage{graphicx}
\usepackage[dvipsnames]{xcolor}
\usepackage{multirow}
\usepackage[nocompress]{cite}
\usepackage[square,numbers]{natbib}
\usepackage{authblk}
\usepackage{tikz,pgf}
\usepackage{tkz-euclide}
\usepackage{bm}
\usepackage{eqnarray}
\newtheorem{theorem}{Theorem}
\newtheorem{defn}{Definition}
\newtheorem{lemma}{Lemma}
\newtheorem{proposition}{Proposition}
\newtheorem{cor}{Corollary}
\newtheorem{remark}{Remark}

\newcommand{\RR}{{\mathbb R}}
\newcommand{\CC}{\mathbb{C}}
\newcommand{\per}{\text{per }}

\newcommand{\refpart}[1]{{\it (#1)}}

\definecolor{colB}{HTML}{6699CC}
\definecolor{colbg}{HTML}{23373b}

\tikzstyle{vertex}=[circle, draw, fill=black, inner sep=0pt, minimum size=4pt]
\tikzstyle{edge}=[line width=1.5pt,black!50!white]
\tikzstyle{redge}=[line width=1.5pt,Red]
\tikzstyle{bledge}=[line width=1.5pt,black]
\tikzstyle{bedge}=[line width=1.5pt,NavyBlue,dashed]
\tikzstyle{ovedge}=[line width=1.5pt,OliveGreen]
\tikzstyle{cedge}=[line width=1.5pt, cyan]
\tikzstyle{cvertex}=[circle, thick, draw=colbg, fill=colB, inner sep=0pt, minimum size=5pt]

\newcommand{\Pseudo}{
	\begin{tikzpicture}[scale=.4]
	\node[vertex, label=right:4] (3) at (6.2, 3) {4};
	\node[vertex, label=right:3] (4) at (5.2, -1) {3};
	\node[vertex, label=above:5] (5) at (4.8,1.7) {5};
	\node[vertex,label=below:6] (6) at (3.6,2.75*0.45)  {6};
	\node[vertex,label=left:1] (7) at (-1.1,2.75*0.5) {1};
	\node[vertex, label=left:2] (8) at (-2.3,0.2) {2};
	\draw[edge] (5)edge(6)  (5)edge(3)  ;
	\draw[edge] (6)edge(7) (8)edge(7) (8)edge(4) (3)edge(4) (6)edge(4);
	
	\node (14) at (5.2, -2.5) {};
	\node (27) at (-1.1, 2.75*0.5+1.5) {};
	\node (28) at (-2.3, -1.3) {};
	\node (23) at (4.7, 3) {};
	\node (25) at (4.8, 0.3) {};
	
	\draw [-to,shorten >=-1pt,red, very thick] (5.25,3) to (5.22,3);
	\draw[redge] (3) to (23);

	\draw [-to,shorten >=-1pt,red, very thick] (5.2, -1.92) to (5.2, -1.95);
	\draw[redge] (4) to (14);
	
	\draw [-to,shorten >=-1pt,red, very thick]  (4.8, 0.85) to (4.8, 0.82);
	\draw[redge] (5) to (25);
	
	\draw [-to,shorten >=-1pt,red, very thick]  (-1.1,2.75*0.5+0.92) to (-1.1,2.75*0.5+0.95) ;
	\draw[redge] (7) to (27);

	\draw [-to,shorten >=-1pt,red, very thick] (-2.3,-0.73) to (-2.3,-0.75);
	\draw[redge] (8) to (28);
	
	\end{tikzpicture}
}

\newcommand{\DisCon}{
	\begin{tikzpicture}[scale=0.3]
	\node[vertex] (1f) at (0, 0) {};
	\node[vertex] (2f) at (0, 6) {};
	
	\node[vertex] (11t) at (1, 1.5) {};      
	\node[vertex] (12t) at (1,4.5) {};  
	\node[vertex] (13t) at (4.5,3) {}; 
	
	\node[vertex] (21t) at (10.5, 1.5) {};      
	\node[vertex] (22t) at (10.5,4.5) {};  
	\node[vertex] (23t) at (7,3) {};
	\node[vertex] (24t) at (14,3) {};  
	
	\draw[bedge] (1f) to (2f);
	
	\draw[redge] (1f) to (11t);
	\draw[redge] (1f) to (21t);
	\draw[redge] (1f) to (13t);	
	
	\draw[redge] (2f) to (12t);
	\draw[redge] (2f) to (22t);
	\draw[redge] (2f) to (23t);
	
	\draw[edge] (11t) to (12t);
	\draw[edge] (11t) to (13t);
	\draw[edge] (13t) to (12t);	
	
	\draw[edge] (21t) to (22t);
	\draw[edge] (21t) to (23t);
	\draw[edge] (23t) to (22t);	
	\draw[edge] (21t) to (24t);
	\draw[edge] (24t) to (22t);	
	
	\draw[colB,very thick,->] (16,3) -- (18,3);
	
	\begin{scope}[xshift=20cm]
	
	\node[vertex] (11t) at (1, 1.5) {};      
	\node[vertex] (12t) at (1,4.5) {};  
	\node[vertex] (13t) at (4.5,3) {}; 
	
	\draw[edge] (11t) to (12t);
	\draw[edge] (11t) to (13t);
	\draw[edge] (13t) to (12t);	
	
	\draw[redge] (11t) to (1,0.5);
	\draw [-to,shorten >=-1pt,red, very thick] (1,.82) to (1,.8);
	
	\draw[redge] (12t) to (1,5.5);
	\draw [-to,shorten >=-1pt,red, very thick] (1,5.2) to (1,5.22);
	
	\draw[redge] (13t) to (5.5,3);
	\draw [-to,shorten >=-1pt,red, very thick] (5.2,3) to (5.22,3);
	
	\end{scope}
	
	\begin{scope}[xshift=29.5cm]
	
	\node[vertex] (11t) at (1, 1.5) {};      
	\node[vertex] (12t) at (1,4.5) {};  
	\node[vertex] (13t) at (-2.5,3) {}; 
	\node[vertex] (14t) at (4.5,3) {}; 
	
	\draw[edge] (11t) to (12t);
	\draw[edge] (11t) to (13t);
	\draw[edge] (13t) to (12t);	
	\draw[edge] (11t) to (14t);
	\draw[edge] (12t) to (14t);
	
	\draw[redge] (11t) to (1,0.5);
	\draw [-to,shorten >=-1pt,red, very thick] (1,.82) to (1,.8);
	
	\draw[redge] (12t) to (1,5.5);
	\draw [-to,shorten >=-1pt,red, very thick] (1,5.2) to (1,5.22);
	
	\draw[redge] (13t) to (-3.5,3);
	\draw [-to,shorten >=-1pt,red, very thick] (-3.2,3) to (-3.22,3);
	
	\end{scope}
	
	\end{tikzpicture}
}

\newcommand{\LDiff}{
	\begin{tikzpicture}[scale=0.3]
	
	\node[vertex] (1) at (-2.8, -1){};
	\node[vertex] (2) at (-2.8, 3) {};
	\node[vertex] (5) at (5.2, 3) {};
	\node[vertex] (6) at (5.2, -1) {};
	\node[vertex] (3) at (1.2,0.1) {};
	\node[vertex] (4) at (1.2,1.9)  {};
	
	\draw[bedge] (1) to (2);
	\draw[redge]   (2)edge(3)  (1)edge(4)   (1)edge(6)  (2)edge(5) ;
	\draw[edge] (3)edge(4)  (3)edge(6)  ;
	\draw[edge] (4)edge(5) (5)edge(6)  ;
	
	\begin{scope}[yshift=6cm]
	
	\node[vertex] (1) at (-2.8, -1){};
	\node[vertex] (2) at (-2.8, 3) {};
	\node[vertex] (3) at (5.2, 3) {};
	\node[vertex] (4) at (5.2, -1) {};
	\node[vertex] (5) at (1.2,0.3) {};
	\node[vertex] (6) at (1.2,1.7)  {};
	
	\draw[bedge] (1) to (2);
	\draw[redge]   (2)edge(3)  (1)edge(4)   (1)edge(5)  (2)edge(6) ;
	\draw[edge] (5)edge(6)  (5)edge(4)  ;
	\draw[edge] (6)edge(3) (3)edge(4)  ;
	
	\end{scope}
	
	\draw[colB,very thick,->] (7,1.5) -- (9.5,3.3);
	\draw[colB,very thick,->] (7,6) -- (9.5,4.2);
	
	\begin{scope}[yshift=3cm,xshift=13.5cm,scale=0.8]
	\node[vertex] (3) at (-2.8, -1){};
	\node[vertex] (4) at (-2.8, 3) {};
	\node[vertex] (5) at (5.2, 3) {};
	\node[vertex] (6) at (5.2, -1) {};
	
	\draw[edge] (3)edge(4)  (4)edge(5)  ;
	\draw[edge] (6)edge(5) (6)edge(3);
	
	\draw [-to,shorten >=-1pt,red, very thick] (-2.8, -2.1) to (-2.8,-2.13);
	\draw[redge] (3) to (-2.8, -2.5);
	
	\draw [-to,shorten >=-1pt,red, very thick] (-2.8, 4.1) to (-2.8,4.13);
	\draw[redge] (4) to (-2.8, 4.5);
	
	\draw [-to,shorten >=-1pt,red, very thick] (5.2, 4.1) to (5.2,4.13);
	\draw[redge] (5) to (5.2, 4.5);
	
	\draw [-to,shorten >=-1pt,red, very thick] (5.2, -2.1) to (5.2,-2.13);
	\draw[redge] (6) to (5.2, -2.5);
	
	\end{scope}
	
	\end{tikzpicture}
}

\newcommand{\Destr}{
	\begin{tikzpicture}[scale=0.3]
	
	\node[vertex] (1) at (-2.8, -1){};
	\node[vertex] (2) at (-2.8, 3) {};
	\node[vertex] (3) at (5.2, 3) {};
	\node[vertex] (4) at (5.2, -1) {};
	\node[vertex] (5) at (1.2,0.3) {};
	\node[vertex] (6) at (1.2,1.7)  {};
	
	\draw[bedge] (2) to (3);
	\draw[redge]   (2)edge(1)    (2)edge(6) (6)edge(3) (3)edge(4);
	\draw[edge] (5)edge(6)  (5)edge(4)  ;
	\draw[edge] (1)edge(4)   (1)edge(5)  ;
	
	\draw[colB,very thick,->] (1.2,-1.5) -- (1.2,-3.3);
	
	\begin{scope}[yshift=-6.5cm]
	\node[vertex] (1) at (-2.8, -1){};
	\node[vertex] (4) at (5.2, -1) {};
	\node[vertex] (5) at (1.2,0.3) {};
	\node[vertex] (6) at (1.2,1.7)  {};
	
	\draw[edge] (5)edge(6)  (5)edge(4)  ;
	\draw[edge] (1)edge(4)   (1)edge(5)  ;
	
	\draw [-to,shorten >=-1pt,red, very thick] (-2.8, -2.1) to (-2.8,-2.13);
	\draw[redge] (1) to (-2.8, -2.5);
	
	\draw [-to,shorten >=-1pt,red, very thick] (5.2, -2.1) to (5.2,-2.13);
	\draw[redge] (4) to (5.2, -2.5);
	
	\draw [-to,shorten >=-1pt,red, very thick] (0.1,1.7) to (0.07,1.7);
	\draw[redge] (6) to (-0.3,1.7);
	\draw [-to,shorten >=-1pt,red, very thick] (2.3,1.7) to (2.33,1.7);
	\draw[redge] (6) to (2.7,1.7);

	\end{scope}
	\end{tikzpicture}
}

\newcommand{\EliiiH}{
	\begin{tikzpicture}[scale=0.3]
	
	\node[vertex] (2) at (1, 1) {};
	\node[vertex] (3) at (-1, -2) {};      
	\node[cvertex] (4) at (0,0.0) {};  
	
	\draw[edge] (2) to (4);
	\draw[edge] (3) to (4);
	\draw[gray!60,densely dashed] (0,0) circle (2.5);
	
	\draw [-to,shorten >=-1pt,red, very thick] (-0.75,0.75) to (-.78,.78);
	\draw[redge] (4) to (-1,1);
	
	\draw[colB,very thick,->] (3,0.5) -- (5,1.5);
	\draw[colB,very thick,->] (3,-0.5) -- (5,-1.5);	
	
	\begin{scope}[xshift=7.5cm]
	
	\draw[gray!60,densely dashed] (0,2) circle (1.5);
	
	\node[vertex] (w2) at (0.5, 2.5) {};
	\node[vertex] (w3) at (-0.5, 1) {};      
	
	\draw [-to,shorten >=-1pt,red, very thick] (-0.1,1.9) to (-.12,1.88);
	\draw[redge] (w2) to (-0.3,1.7);
	
	\draw[gray!60,densely dashed] (0,-2) circle (1.5);
	
	\node[vertex] (z2) at (0.5, -1.5) {};
	\node[vertex] (z3) at (-0.5, -3) {};      
	
	\draw[redge] (z3) to (0.3,-2.2);
	\draw [-to,shorten >=-1pt,red, very thick] (0.1,-2.42) to (0.12,-2.4);
	
	\end{scope}
	\end{tikzpicture}
}

\newcommand{\Eliii}{
	\begin{tikzpicture}[scale=0.3]
	\node[vertex] (1) at (-1.5, 1.5) {};
	\node[vertex] (2) at (1, 1) {};
	\node[vertex] (3) at (-1, -2) {};      
	\node[cvertex] (4) at (0,0.0) {};  
	\draw[edge] (1) to (4);
	\draw[edge] (2) to (4);
	\draw[edge] (3) to (4);
	\draw[gray!60,densely dashed] (0,0) circle (2.5);
	
	\draw[colB,very thick,->] (3,0) -- (5,0);
	\draw[colB,very thick,->] (3,2) -- (5,3);
	\draw[colB,very thick,->] (3,-2) -- (5,-3);

	\begin{scope}[xshift=7.5cm]
	\draw[gray!60,densely dashed] (0,0) circle (1.5);
	\node[vertex] (u1) at (-.7, .7) {};
	\node[vertex] (u2) at (0.5, 0.5) {};
	\node[vertex] (u3) at (-0.5, -1) {};      
	
	\draw [-to,shorten >=-1pt,red, very thick] (-0.2,.2) to (-0.18,.18);
	\draw[redge] (u1) to (-0,0);
	
	\end{scope}
	
	\begin{scope}[xshift=7.5cm, yshift=4cm]
	\draw[gray!60,densely dashed] (0,0) circle (1.5);
	\node[vertex] (u1) at (-.7, .7) {};
	\node[vertex] (u2) at (0.5, 0.5) {};
	\node[vertex] (u3) at (-0.5, -1) {};      
	
	\draw [-to,shorten >=-1pt,red, very thick] (-0.02,-.02) to (-0.04,-.04);
	\draw[redge] (u2) to (-0.2,-0.2);
	
	\end{scope}
	
	\begin{scope}[xshift=7.5cm, yshift=-4cm]
	\draw[gray!60,densely dashed] (0,0) circle (1.5);
	\node[vertex] (u1) at (-.7, .7) {};
	\node[vertex] (u2) at (0.5, 0.5) {};
	\node[vertex] (u3) at (-0.5, -1) {};      
	
	\draw [-to,shorten >=-1pt,red, very thick] (0.02,-.48) to (0.04,-.46);
	\draw[redge] (u3) to (0.2,-0.3);
	
	\end{scope}

	\end{tikzpicture}
}

\newcommand{\Elpath}{
	\begin{tikzpicture}[scale=0.27]
	
	\node[vertex] (b) at (3, 2) {};
	\node[vertex] (a) at (-2, -3) {};  
	\node[cvertex] (1) at (-1,-1) {};      
	\node[cvertex] (2) at (0,0.0) {};  
	\node[cvertex] (3) at (1,1) {};  

	\draw[edge] (1) to (2);
	\draw[edge] (2) to (3);

	\draw[edge] (1) to (a);
	\draw[edge] (3) to (b);
	\draw[gray!60,densely dashed] (0,0) circle (4.0);
	
	\draw [-to,shorten >=-1pt,red, very thick] (-1.78,-0.22) to (-1.8,-.2);
	\draw[redge] (1) to (-2,0);
	
	\draw [-to,shorten >=-1pt,red, very thick] (-0.78,.78) to (-.8,.8);
	\draw[redge] (2) to (-1,1);
	
	\draw [-to,shorten >=-1pt,red, very thick] (0.22,1.78) to (0.2,1.8);
	\draw[redge] (3) to (0,2);
	
	\draw[colB,very thick,->] (4.5,0.5) -- (6.5,1.5);
	\draw[colB,very thick,->] (4.5,-0.5) -- (6.5,-1.5);	
	
	\begin{scope}[xshift=9.5cm,yshift=3cm, scale=2.5/4]
	
	\node[vertex] (b) at (3, 2) {};
	\node[vertex] (a) at (-1, -3.5) {};  
	\node[cvertex] (1) at (-1,-1) {}; 
	\node[cvertex] (3) at (1,1) {};  	
	
	\draw[edge] (1) to (3);
	\draw[edge] (3) to (b);
	\draw[gray!60,densely dashed] (0,0) circle (4.0);
	
	\draw [-to,shorten >=-1pt,red, very thick] (-1,-2.32) to (-1,-2.3);
	\draw[redge] (a) to (-1,-2);
	
	\draw [-to,shorten >=-1pt,red, very thick] (-2.18,0.18) to (-2.2,.2);
	\draw[redge] (1) to (-2.5,0.5);

	\draw [-to,shorten >=-1pt,red, very thick] (-0.18,2.18) to (-0.2,2.2);
	\draw[redge] (3) to (-0.5,2.5);
	
	\draw[colB,very thick,->] (5,0) -- (5+13/5,0);
	
	\end{scope}
	
	\begin{scope}[xshift=9.5cm,yshift=-3cm, scale=2.5/4]
	
	\node[vertex] (b) at (3, 2) {};
	\node[vertex] (a) at (-2, -3) {};  
	\node[cvertex] (1) at (-1,-1) {}; 
	\node[cvertex] (3) at (1,1) {};

	\draw[edge] (1) to (3);
	\draw[edge] (3) to (b);
	
	\draw[gray!60,densely dashed] (0,0) circle (4.0);
	
	\draw [-to,shorten >=-1pt,red, very thick] (-2.18,0.18) to (-2.2,.2);
	\draw[redge] (1) to (-2.5,0.5);
	
	\draw [-to,shorten >=-1pt,red, very thick] (0.18,-2.18) to (.2,-2.2);
	\draw[redge] (1) to (.5,-2.5);

	\draw [-to,shorten >=-1pt,red, very thick] (-0.18,2.18) to (-0.2,2.2);
	\draw[redge] (3) to (-0.5,2.5);
	
	\draw[colB,very thick,->] (5,0) -- (5+13/5,0);
	\end{scope}

	\begin{scope}[xshift=17.5cm,yshift=3cm, scale=2.5/4]
	
	\node[vertex] (b) at (3, 2) {};
	\node[vertex] (a) at (-1, -3.5) {};  
	\node[cvertex] (2) at (0,0.0) {};   

	\draw[edge] (2) to (b);
	\draw[gray!60,densely dashed] (0,0) circle (4.0);
	
	\draw [-to,shorten >=-1pt,red, very thick] (-1,-1.92) to (-1,-1.9);
	\draw[redge] (a) to (-1,-1.5);
	
	\draw [-to,shorten >=-1pt,red, very thick] (-1.18,1.18) to (-1.2,1.2);
	\draw[redge] (2) to (-1.5,1.5);
	
	\draw[colB,very thick,->] (5,0) -- (5+13/5,0);
	\end{scope}
	
	\begin{scope}[xshift=17.5cm,yshift=-3cm, scale=2.5/4]
	
	\node[vertex] (b) at (3, 2) {};
	\node[vertex] (a) at (-2, -3) {};  
	\node[cvertex] (2) at (0,0.0) {};

	\draw[edge] (2) to (b);
	
	\draw[gray!60,densely dashed] (0,0) circle (4.0);
	
	\draw [-to,shorten >=-1pt,red, very thick] (-1.18,1.18) to (-1.2,1.2);
	\draw[redge] (2) to (-1.5,1.5);
	
	\draw [-to,shorten >=-1pt,red, very thick] (1.18,-1.18) to (1.2,-1.2);
	\draw[redge] (2) to (1.5,-1.5);

	\draw[colB,very thick,->] (5,0) -- (5+13/5,0);
	\end{scope}
	
	\begin{scope}[xshift=25.5cm,yshift=3cm, scale=2.5/4]
	
	\node[vertex] (b) at (3.5, 1) {};
	\node[vertex] (a) at (-1, -3.5) {};  
	\draw[gray!60,densely dashed] (0,0) circle (4.0);

	\draw [-to,shorten >=-1pt,red, very thick] (-1,-1.92) to (-1,-1.9);
	\draw[redge] (a) to (-1,-1.5);
	
	\end{scope}
	
	\begin{scope}[xshift=25.5cm,yshift=-3cm, scale=2.5/4]
	
	\node[vertex] (b) at (3.5, 1) {};
	\node[vertex] (a) at (-1, -3.5) {};  
	\draw[gray!60,densely dashed] (0,0) circle (4.0);

	\draw [-to,shorten >=-1pt,red, very thick] (1.92,1) to (1.9,1);
	\draw[redge] (b) to (1.5,1);
	
	\end{scope}
	
	\end{tikzpicture}
}

\newcommand{\Elpathd}{
	\begin{tikzpicture}[scale=0.27]
	
	\node[vertex] (b) at (3, 2) {};
	\node[vertex] (a) at (-2, -3) {};  
	\node[cvertex] (1) at (-1,-1) {};      
	
	\node[cvertex] (3) at (1,1) {};

	\draw[edge] (1) to (3);
	\draw[edge] (1) to (a);
	\draw[edge] (3) to (b);
	\draw[gray!60,densely dashed] (0,0) circle (4.0);

	\draw [-to,shorten >=-1pt,red, very thick] (-1.78,-0.22) to (-1.8,-.2);
	\draw[redge] (1) to (-2,0);

	\draw [-to,shorten >=-1pt,red, very thick] (-.22,-1.78) to (-.2,-1.8);
	\draw[redge] (1) to (0.,-2.);

	\draw [-to,shorten >=-1pt,red, very thick] (0.22,1.78) to (0.2,1.8);
	\draw[redge] (3) to (0,2);
	
	\draw [-to,shorten >=-1pt,red, very thick] (1.78,0.22) to (1.8,0.2);
	\draw[redge] (3) to (2,0);
	
	\draw[colB,very thick,->] (4.5,0.5) -- (6.5,1.5);
	\draw[colB,very thick,->] (4.5,-0.5) -- (6.5,-1.5);	
	
	\begin{scope}[xshift=9.5cm,yshift=3cm, scale=2.5/4]
	
	\node[vertex] (b) at (3.5, 1) {};
	\node[vertex] (a) at (-1, -3.5) {};  
	\draw[gray!60,densely dashed] (0,0) circle (4.0);

	\draw [-to,shorten >=-1pt,red, very thick] (-1,-1.92) to (-1,-1.9);
	\draw[redge] (a) to (-1,-1.5);
	
	\end{scope}
	
	\begin{scope}[xshift=9.5cm,yshift=-3cm, scale=2.5/4]
	
	\node[vertex] (b) at (3.5, 1) {};
	\node[vertex] (a) at (-1, -3.5) {};  
	\draw[gray!60,densely dashed] (0,0) circle (4.0);

	\draw [-to,shorten >=-1pt,red, very thick] (1.92,1) to (1.9,1);
	\draw[redge] (b) to (1.5,1);

	\end{scope}

	\end{tikzpicture}
}

\newcommand{\Gbc}{
	\begin{tikzpicture}[scale=0.43]
	
	\node[vertex] (1) at (0, 0) {};
	\node[cvertex] (2) at (0, 1.5) {};
	\node[vertex] (3) at (0, 3) {};
	\node[cvertex] (4) at (0, 4.5) {};
	\node[vertex] (5) at (1.5, 6) {};
	\node[vertex] (6) at (-1.5, 6) {};
	
	\node at (-0.8, 0.1) {$B_{1}$};
	\node at (-0.8, 1.5) {$v_{c}$};
	\node at (-0.8, 3) {$l_{u,c}$};
	\node at (-0.8, 4.5) {$u_{c}$};
	\node at (1.6,6.7) {$B_{2}$};
	\node at (-1.6, 6.7) {$B_{3}$};

	\draw[edge] (1) to (2);
	\draw[edge] (2) to (3);
	\draw[edge] (3) to (4);
	\draw[edge] (4) to (6);
	\draw[edge] (5) to (4);

	\end{tikzpicture}
}

\newcommand{\bcG}{
	\begin{tikzpicture}[scale=0.47]
	
	\node[vertex] (1) at (-1, 0) {};
	\node[cvertex] (2) at (1, 0) {};
	\node[vertex] (3) at (1, 2) {};
	
	\node at (0.3, 0.45) {$B_{1}$};
	\node at (1, -0.63) {$v_{c}$};
	
	\draw [-to,shorten >=-1pt,red, very thick] (-1,-0.73) to (-1,-.75);
	\draw[redge] (1) to (-1,-1);
	\draw [-to,shorten >=-1pt,red, very thick] (1,2.73) to (1,2.75);
	\draw[redge] (3) to (1,3);
	
	\draw[edge] (1) to (2);
	\draw[edge] (1) to (3);
	\draw[edge] (2) to (3);

	\node[cvertex] (4) at (4, 0.3) {};
	\node[vertex] (5) at (6.5, 0.3) {};
	\node[vertex] (6) at (6.4, 2.2) {};
	\node[vertex] (7) at (3.9, 2.4) {};
	
	\node at (2.5, 0.75) {$l_{u,c}$};
	\node at (3.5, -0.25) {$u_c$};
	\node at (5, 1.3) {$B_2$};
	
	\draw[edge] (2) to (4);
	\draw[edge] (4) to (5);
	\draw[edge] (5) to (6);
	\draw[edge] (7) to (6);
	\draw[edge] (7) to (4);

	\draw [-to,shorten >=-1pt,red, very thick] (3.9,3.13) to (3.9,3.15);
	\draw[redge] (7) to (3.9,3.4);
	
	\draw [-to,shorten >=-1pt,red, very thick] (6.4,2.95) to (6.4,2.97);
	\draw[redge] (6) to (6.4,3.2);
	
	\draw [-to,shorten >=-1pt,red, very thick] (6.5,-0.45) to (6.5,-0.457);
	\draw[redge] (5) to (6.5,-0.7);
	
	\node[vertex] (8) at (6.2, -2) {};
	\node[vertex] (9) at (4.1, -2) {};
	
	\node at (4.7, -1.2) {$B_3$};
	
	\draw [-to,shorten >=-1pt,red, very thick] (6.2,-2.75) to (6.2,-2.77);
	\draw[redge] (8) to (6.2,-3);
	\draw [-to,shorten >=-1pt,red, very thick] (4.1,-2.75) to (4.1,-2.77);
	\draw[redge] (9) to (4.1,-3);
	
	\draw[edge] (8) to (9);
	\draw[edge] (4) to (8);
	\draw[edge] (9) to (4);

	\end{tikzpicture}
}

\title{New upper bounds for the number of embeddings of minimally rigid graphs}

\author[1,2]{Evangelos Bartzos}
\author[1,2]{Ioannis Z. Emiris}
\author[3]{Raimundas Vidunas}

\affil[1]{Department of Informatics and Telecommunications, National Kapodistrian University of Athens}

\affil[2]{ATHENA Research Center}
\affil[3]{Institute of Applied Mathematics, Faculty of Mathematics and Informatics, Vilnius University}

\date{}

\begin{document}
	\maketitle
\begin{abstract}
  By definition, a rigid graph in $\mathbb{R}^d$ (or on a sphere) has a finite number of embeddings up to rigid motions for a given set of edge length constraints.
  These embeddings are related to the real solutions of an algebraic system.
  Naturally, the complex solutions of such systems extend the notion of rigidity to $\CC^d$.
  A major open problem has been to obtain tight upper bounds on the number of embeddings in $\CC^d$, for a given number $|V|$ of vertices, which obviously also bound their number in $\RR^d$.
  Moreover, in most known cases, the maximal numbers of embeddings in $\CC^d$ and $\RR^d$ coincide.
  For decades, only the trivial bound of $O(2^{d\cdot |V|})$ was known on the number of embeddings.
  Recently, matrix permanent bounds have led to a small improvement for $d\geq 5$.
  
  This work improves upon the existing upper bounds for the number of embeddings in $\RR^d$ and $S^d$, by exploiting outdegree-constrained orientations on a graphical construction, where the proof iteratively eliminates vertices or vertex paths.
  For the most important cases of $d=2$ and $d=3$, the new bounds are $O(3.7764^{|V|})$ and $O(6.8399^{|V|})$, respectively.
  In general, the recent asymptotic bound mentioned above is improved by a factor of $1/ \sqrt{2}$. 
  Besides being the first substantial improvement upon a long-standing upper bound, our method is essentially the first general approach relying on combinatorial arguments rather than algebraic root counts.  
  
\end{abstract}

\section{Introduction}\label{sec:intro}

Rigid graph theory is the branch of mathematics examining the properties of graphs that admit rigid embeddings.
Even if the origins of rigidity theory are rather old \cite{Maxwell}, it receives nowadays significant attention motivated by applications in robotics \cite{Rob1,Drone}, molecular biology \cite{Bio2,Bio1}, sensor network localization \cite{sensor}, and architecture \cite{arch2,arch1}.
Moreover, rigidity theory, and the closely related field of distance geometry, are also relevant as independent mathematical areas at the intersection of graph theory, computer algebra, and computational geometry.

Let $G=(V,E)$ be a graph with vertex and edge sets denoted by $V$ and $E$, respectively.
Let $\mathbf{p} = \{ p_1, \dots, p_{|V|} \} \in \RR^{d |V|}$ be an embedding of the graph in a Euclidean space $\RR^d$, in other words the assignment of the vertices to points, once a coordinate frame has been fixed.
Every embedding induces a set of edge lengths $\{ \lVert p_u-p_v \rVert \}_{(u,v) \in E}$,
where $\lVert \cdot \rVert$ denotes Euclidean norm.
This embedding is called \emph{rigid} if the total number of all graph embeddings that satisfy the same edge length constraints is finite modulo rigid transformations; otherwise the embedding is called \emph{flexible}.

Embeddings correspond to real solutions of an algebraic system that captures the edge length constraints. 
Therefore, the complex roots of these systems extend the notion of rigidity to the complex space, thus we may refer to \emph{complex embeddings}.
The number of embeddings in $\RR^d$ is thus bounded by the number of real or complex solutions of the algebraic system, or by bounds on these numbers. The most relevant bounds on the number of complex solutions of well-constrained algebraic systems are B\'ezout's bound, the multivariate B\'ezout bound (m-B\'ezout)~\cite{Shafarevich2013}, and the Bernstein (also known as BKK) bound~\cite{Bernshtein1975,Khovanskii1978,Kouchnirenko1976} expressed by the system's mixed volume.

One of the most important theorems in rigidity theory states that rigidity is a \textit{generic property} \cite{AsimowRoth}.
This means that a graph is either rigid or flexible for an open dense subset of embeddings $\mathbf{p}\in \RR^{d |V|}$; any such embedding shall be called a \textit{generic embedding} \cite{handbook1}.
If any edge deletion in a generically rigid graph $G$ yields a flexible graph, then $G$ is a \textit{generically minimally rigid graph}.
Let us note that every graph that is generically minimally rigid in $\CC^d$ is also generically minimally rigid on the sphere $S^d$ \cite{Whiteley_cone}.
Generically minimally  rigid graphs in $\CC^2$  are better known as \emph{Laman graphs} \cite{tay,handbook1} while, following \cite{GKT17,belt,bes} we will call minimally  rigid graphs in $\CC^3$ \emph{Geiringer graphs}.

Another well-known theorem due to Maxwell~\cite{Maxwell} gives necessary conditions for minimal rigidity that are related to a simple edge count in the graph and its subgraphs.
It states that, if $G=(V,E)$ is a minimally generically rigid graph in $\CC^d$, then $|E|=d\dot |V|- \binom{d+1}{2}$ and, for each subgraph $G'=(V',E')$ of $G$ with $|V'|\geq d$, it holds that $|E'|\leq d\dot |V'|-\binom{d+1}{2}$.
The total edge count corresponds intuitively to the number of vertex coordinates reduced by the number of degrees of freedom of rigid motions.
This condition is also sufficient when $d=2$ for $G$ to be Laman \cite{Geiringer1927,Geiringer1932,Laman}.
Minimally rigid graphs in any dimension are related to well-constrained algebraic systems: the number of complex embeddings is invariant of the choice of (generic) edge lengths \cite{Jackson}.

A major open question in rigidity theory is to establish tight upper bounds on the number of embeddings of generically minimally rigid graphs and especially Laman and Geiringer graphs, see e.g., \cite{Borcea,Jackson,bes}.
This has been so far addressed by bounding the root count in the corresponding algebraic system. A straightforward bound of $O(2^{d|V|})$ thus follows from B\'ezout's bound  derived on a system of $O(d|V|)$ quadratic length equations (in the rest of the paper whenever we refer to the \emph{B\'ezout bound} will be the above-mentioned bound on the number of embeddings).

A more sophisticated method relies on determinantal equations (and inequalities) from the minors of the Cayley-Menger matrix \cite{Blu}, a well-known extended squared-distance matrix.
By applying a theorem on the degree of determinantal varieties~\cite{HarrisTu}, Borcea and Streinu presented an upper bound which has been the best available for a long time \cite{Borcea}, namely
$$
2 \cdot \prod \limits_{m=0}^{|V|-d-2}  {\displaystyle\binom{|V|-1+m}{|V|-d-1-m}} \, / \, {\displaystyle\binom{2m+1}{m}} .
$$
However, this does not improve asymptotically upon the 
trivial B\'ezout bound.
Lastly, certain mixed volume techniques introduced for $d=2$ in~\cite{Steffens} did not manage to improve upon the latter bound.

Recently, a different approach that combines m-B\'ezout bounds with bounds on the value of matrix permanents \cite{Bre73,Minc63} has improved, for the first time, upon the trivial bounds~\cite{bes} (see Section~\ref{sec:mbez} for details), albeit only for $d\geq 5$; in fact, the improvement is significant for substantially larger dimension.
Table~\ref{tab:bounds} juxtaposes the classic B\'ezout bound to this improvement and the results of the present paper; actual bounds are powers of the given basis to the power $|V|-d$.
Unlike all aforementioned results, the approach introduced in this paper exploits certain novel combinatorial constructions rather than algebraic root counts exploiting the structure of the underlying equations.

\begin{table}[htp!]
	\caption{Power basis of upper bounds for minimally rigid graphs in $\CC^d$: 
		the first line contains the bounds derived in this paper,
		B-M those from Br\`egman-Minc \cite{bes}, and B\'ez.\ corresponds to the trivial B\'ezout bound.
		\label{tab:bounds}}
	\begin{tabular}{l|llllllll}\hline
		\noalign{\smallskip}
		$d$ & 2 & 3 & 4 &5  & 6 & 7 & 8 & 9 \\
		\hline\\ 
		\emph{this} & 3.7764 &   6.8399  & 12.686  &  23.899 & 45.533 & 87.469 & 168.90 & 327.45 \\
		\emph{B-M}  & 4.8990 & 8.9442 & 16.733 & 31.749 & 60.795 & 117.17 & 226.89 & 441\\
		\emph{B\'ez.}  &  4 &   8  &  16 & 32 & 64  &  128 & 256 & 512 \\
		\noalign{\smallskip}\hline
	\end{tabular}
\end{table}

Furthermore, in~\cite{bes}, directed graphs were considered.
In particular, graph orientation bounds derived from~\cite{Felsner} were employed to show that there are at most $O\left(7.1131^{|V|}\right)$ embeddings for the subclass of planar (in the graph-theoretical sense) Geiringer graphs. 
This has been a minor improvement in the long quest for better upper bounds, but it was the first to rely on combinatorial rather than algebraic arguments, thus paving the way to the approach of this paper, as presented in Section~\ref{sec:mbez}.

Asymptotic lower bounds are computed by gluing smaller graphs with known number of embeddings \cite{Borcea,GKT17,etv}, in particular graphs with maximal number of embeddings among all minimally rigid graphs in a given embedding space with same vertex number.
For Laman and Geiringer graphs, there are huge gaps between upper and lower bounds:
for Laman graphs there are $\Omega(2.5080^{|V|})$ 
embeddings in $\CC^2$ \cite{Joseph_lam,GKT17}, $\Omega(2.5698^{|V|})$ complex embeddings on the sphere \cite{count_sphere}, whereas for Geiringer graphs the bound becomes $\Omega(3.0683^{|V|})$ in $\CC^3$ \cite{GKT17}.

\paragraph{Our contribution.} 
We extend the method based on outdegree-constrained orientations, introduced in~\cite{bes}, in order to establish new upper bounds on the embedding numbers of minimally rigid graphs.
These bounds are essentially the first to improve upon the trivial B\'ezout bounds for every dimension $d \geq 2$.

The new bounds improve upon all existing bounds by relying on certain novel combinatorial constructions; in contrast, we do not exploit any further structure in the underlying algebraic system.
First, we introduce a graph structure that inherits some of the properties of minimally rigid graphs, which we call \textit{pseudographs}.
Then, we apply an iterative method that eliminates a vertex or a path in each step, while maintaining some basic properties of the pseudograph.
This is used to derive bounds on fixed-outdegree orientations of pseudographs, which eventually improve upon the current upper bound on the embedding numbers. 

In particular, we improve the bound for Laman graphs to $O(3.7764^{|V|})$ and that for Geiringer graphs to $O(6.8399^{|V|})$.
The first improvement solves \textit{Open Problem 1} in \cite{GKT17}.
For higher dimensions, the improvement over the asymptotic formula presented in~\cite{bes} is by a factor of $1/ \sqrt 2$.
Let us note that the bounds in this paper hold both for embeddings in $\CC^d$ as well as in the $d-$dimensional sphere $S^d$.
Thus, whenever we refer to upper bounds on embedding numbers, both spaces are implied.

\paragraph{Organization.} 
The rest of the paper is organized as follows. 
In Section~\ref{sec:mbez} we present a recent result \cite{bes} that relates graph orientations to the m-B\'ezout bound applied to bounding the embeddings of minimally rigid graphs.
In Section~\ref{sec:lam}, we introduce a graphical structure and subsequently we compute a bound on the number of its orientations with fixed outdegree $2$.
This leads to a new upper bound on the embeddings of Laman graphs.
In Section~\ref{sec:higher}, we generalize this method to the case of orientations with fixed outdegree $d$ and derive improved upper bounds for higher dimensions.
Finally, in Section~\ref{sec:conc} we conclude and present some open questions.

\section{Edge orientations and bounds on embeddings}\label{sec:mbez}

This section discusses two results from~\cite{bes}, used in the present paper.
Both rely on well-constrained algebraic systems whose solutions correspond to embeddings.
In the sequel (unless specified otherwise), when giving bounds on the embedding number, we refer to bounds on the number of algebraic roots.

The first result, which shall be extended in this paper, relates embedding numbers to the number of graph orientations that satisfy certain constrains on the outdegree of vertices. 
Graph orientations have already been applied in rigidity theory, e.g.~\cite{pebble2,Whiteley,Tibor}.
Our starting point is the following theorem\footnote{In~\cite{bes}, indegree constraints were used but here, following~\cite{pebble2}, we use (equivalently) outdegrees.}:

\begin{theorem}[{}\cite{bes}]\label{thm:bez}
	Let $G(V,E)$ be a minimally rigid graph in $\CC^d$ that contains at least one complete subgraph with $d$ vertices.
	Let $K_d=(v_1,\dots v_d)$ be one such subgraph: we call it {\em fixed}, and we also call its vertices and edges {\em fixed}.
	By removing the edges of $K_d$ from $G$, graph $G'=(V,E \backslash\, edges(K_d))$ is defined.
	Let $B(G,K_d)$ denote the number of outdegree-constrained orientations of $G'$, such that 
	\begin{itemize}
		\item the outdegree of $v_1, \dots , v_d$  is $0$.
		\item the outdegree of every vertex in $V\backslash \{v_1,\dots , v_d\}$ is $d$.
	\end{itemize}
	Then, the number of embeddings of $G$ in $\CC^d$ does not exceed	
	$$ 
	2^{|V|-d}\cdot B(G,K_d) .
	$$
\end{theorem}
Minimally rigid graphs have complex embeddings for every generic choice of edge lengths, hence their bound is positive, and so is the number of orientations in Theorem~\ref{thm:bez}.

Figure~\ref{fig:orient} illustrates the theorem for the famous 6-vertex Desargues' graph in the plane (aka double-prism), yielding an upper bound of 32, while the actual embedding number in $\RR^2$ and $\CC^2$ is $24$~\cite{Borcea}. 
In $S^2$ there are 32 embeddings, in the real and complex case~\cite{belt}.

\begin{figure}[htp]
	\begin{center}	
		\includegraphics[scale=0.43]{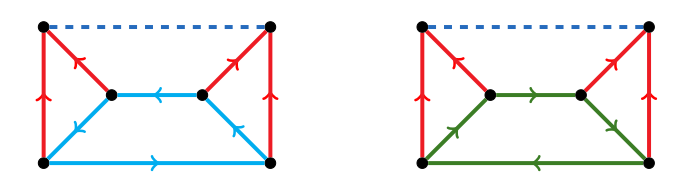}
		\caption{The two possible constrained orientations of Desargues' graph.
			There is a unique orientation of red edges up to the choice of the fixed (dashed blue) edge.
			The number of embeddings is bounded by $2^{6-4} \cdot 2=32$.	\label{fig:orient}}
	\end{center}
\end{figure}

The second relevant result in~\cite{bes} had improved the bound on Euclidean embeddings for $d\geq 5$.
Matrix $A_{G,K_d}$ was constructed by repeating each row of the incidence matrix $d$ times, while excluding the rows that represent the vertices and edges of $K_d$. 
For a rigid graph and a fixed $K_d$, the number of outdegree-constrained orientations (described in Theorem~\ref{thm:bez}) is related to the permanent of $A_{G,K_d}$, denoted by $\per(\cdot)$, as follows:
$$
B(G,K_d)~=~\displaystyle \left(\frac{1}{d!} \right)^{|V|-d} \per(A_{G,K_d}).
$$
By applying the Br\`egman-Minc bound on the permanent of $(0,1)$-matrices \cite{Bre73,Minc63} to $\per(A_{G,K_d})$, one obtains
\begin{equation}\label{eq:bes}
	O\left( \left(2\cdot \sqrt{ \binom{2d}{d}} \right)^{|V|} \right)
\end{equation}
as an upper bound on the number of embeddings, thus improving upon the B\'ezout bound for every $d\geq 5$. 
In this paper, we improve this asymptotic formula by a factor of $1/ \sqrt{2}$, see Equation~(\ref{eq:asym}).

\paragraph{Existence of a clique.}
Let us elaborate on the existence of $K_d$ and on alternatives to compute the bound in its absence.
In Laman graphs, there is always a fixed edge, or $K_2$ (Figure~\ref{fig:orient}).
In most known cases of Geiringer graphs, there is a fixed triangle, but there exists minimally rigid graphs with no triangles: $K_{6,4}$ is the only instance with up to $10$ vertices.

Generally, if for $d\geq 3$ no $K_d$ exists, a maximal clique may be fixed with $d'<d$ vertices and for the rest $d-d'$ vertices one may fix an appropriate number of coordinates, thus factoring out rotations and translations according to Maxwell's condition.
Clearly, $d'\geq 2$ since an edge always exists.

For $d=3$, Maxwell's condition removes $6$ degrees of freedom (dof) and the presence of a triangle removes $3$ additional degrees of freedom, thus fixing the $9$ coordinates of the triangle.
If no triangle exists, $3$ vertices are selected such that $2$ of them are the endpoints of an edge. 
Then, we use Maxwell's condition to remove  $6$ degrees of freedom as follows: first we define a plane on which all three vertices lie by fixing one of their coordinates e.g. $x=0$ for all three, removing $3$ dof, and then we fix the other $2$ coordinates of the first vertex and $1$ more coordinate of the second vertex removing $3$ more dof.
An additional dof is removed using the edge, fixing the third coordinate of the second vertex.
Now the first two vertices are \emph{fixed}, while the third is \emph{partially fixed}.
The corresponding algebraic system counts every embedding twice (by reflection on the plane defined above).

In order to compute the bound, we count orientations such that the two fixed vertices have outdegree $0$, the partially fixed vertex has outdegree $2$ (equal to the degrees of freedom of the vertex), while the others have outdegree $3$.
Let us denote the number of these orientations by $B(G,K_2')$ for a given graph and fixed coordinates.
The bound for the algebraic system is obtained by $2^{|V|-2}\cdot B(G,K_2')$, while for the bound on the embeddings we divide by $2$ to obtain $2^{|V|-3}\cdot B(G,K_2') $.
A similar approach can be applied in higher dimensions.
\footnote{A similar discussion in \cite[Sec.~2.2]{bes} offers details on the algebraic systems.}

\section{Bounding the number of embeddings for Laman graphs} \label{sec:lam}

In this section we develop a method to improve the upper bound on Laman graphs embeddings.
We introduce \emph{pseudographs}, and bound the number of orientations for connected pseudographs with fixed outdegree equal to~$2$.
This applies basic graphical operations that reduce the size of the pseudograph in an iterative process.
Then, we relate the upper bound on orientations to this process.
Finally, we connect this bound to the total bound on embeddings.

\subsection{Pseudographs and orientations with fixed outdegree~$2$.}

We define the following graphical structure generalizing that of a graph.

\begin{defn}\label{def:pseu}
	A {\em{pseudograph}} $L(U,F,H)$ is a collection, where $U$ is a set of vertices, $F$ is a set of edges called {\em{normal edges}}, each incident to two vertices in $U$, and $H$ is a set of edges called {\em{hanging edges}}, each with a single endpoint in $U$ and directed out of the vertex\footnote{Hanging edges are reminiscent of "directed loops" in hypergraphs \cite{pebble2}; "half-edges" also have a single endpoint.}.
	Moreover, the graph $G(U,F)$ is called {\em{normal subgraph}}.
	If the normal subgraph is connected, then $L$ is a {\em{connected pseudograph}}.
\end{defn}

Let the \textit{total degree} $p$ of a vertex $v$ denote the total number of  (normal and hanging) edges incident to $v$.
Let $h$ denote the \textit{hanging degree} of $v$, which is the number of hanging edges incident to $v$, while the number of normal edges incident to $v$ is its \textit{normal degree} and equals $p-h$.
The \textit{extended degree} of $v$ is the pair $(p,h)$.

\begin{figure}[htp]
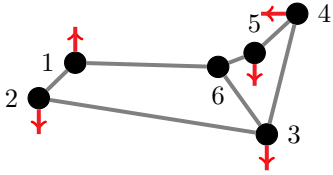

	\centering
	\Pseudo
	\caption{\label{fig:Pseudo} A pseudograph with 6 vertices.
		The extended degrees are the following: $(3,1)$ for vertices $1,2,4,5$, $(3,0)$ for vertex $6$, and $(4,1)$ for vertex $3$.
	}
\end{figure}

We shall consider orientations of a pseudograph $L$ defined by specifying a direction on every normal edge, while by definition hanging edges are directed out of their unique vertex.
Pseudograph orientations refer to the orientations of pseudographs.
We count pseudograph orientations with fixed outdegree $2$ for all vertices: we call these orientations \textit{valid}.
Clearly, if a vertex belongs to a hanging edge, one more edge should be directed out of it, while if it has hanging degree $2$, all its normal edges should be in-directed.
A pseudograph containing a vertex with extended degree $(p,h)$, such that $p<2$ or $h>2$, has no valid orientations.

We now prove the following necessary condition for the existence of a valid orientation of a pseudograph (which resembles Maxwell's count).

\begin{proposition}\label{prop:pseucount}
	Let $L(U,F,H)$ be a pseudograph with a valid orientation.
	Then $|F|+|H|=2|U|$.
\end{proposition}
\begin{proof}
	
	$|F|+|H|$ is the sum of outdegrees over all edges; $2|U|$ equals the sum of outdegrees over vertices. 
\end{proof}

\subsection{Iterative elimination}

Now we present the basic graphical operations used to reduce the size of a connected pseudograph.
We specify an iterative elimination process comprised of a sequence of steps, with the requirement that the pseudograph stays connected.
We shall distinguish two types of steps, depending on the extended degree of the vertex, or of the vertex path to be eliminated.
The process terminates when the current pseudograph's normal subgraph is a tree; see details in Proposition~\ref{prop:tree}.

\begin{figure}[htp!]
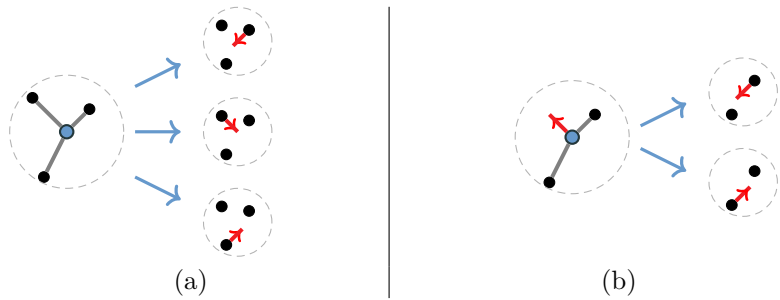
	
	\centering \begin{tabular}{c|c}
		\Eliii \hspace*{12mm}	& \hspace*{12mm} \raisebox{.25\height}{\EliiiH} \\
		(a) & \hspace*{6mm} (b) 
		\vspace{3mm}\\
	\end{tabular} 
	\\
	\caption{\label{fig:el}
		Elimination of a vertex with extended degree (a) $(3,0)$, encountered in vertex elimination, or (b) $(3,1)$, encountered in path elimination.
		In (a) there are $3$ choices for eliminating edges, resulting in 3 different pseudographs; in (b) there are $2$ choices.
	}
\end{figure}

Let us detail the two types of elimination steps.

The first type eliminates a single vertex $v$ with extended degree other than $(3,1)$.
Let $L(U,F,H)$ be a pseudograph:
We choose to eliminate two edges incident to $v$ (Figure~\ref{fig:el}a), thus maintaining the total edge count of Proposition~\ref{prop:pseucount}.
If $v$ is incident to $h\le 2$ hanging edges, these must be eliminated.
Since the outdegree of $v$ equals~2 in a valid orientation of $L$, there are $2-h \le 2$ normal edges incident to $v$ that get eliminated.
All edges that are not eliminated become hanging in the new pseudograph, and correspond to edges directed towards $v$ for a valid orientation of the initial pseudograph.

\begin{figure}[htp]
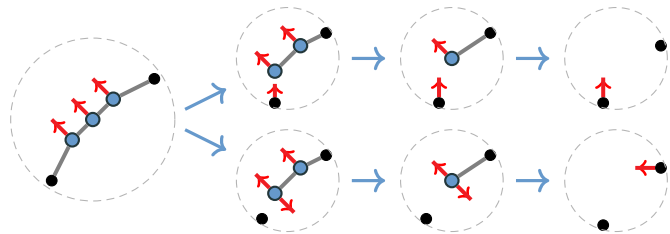

	\centering	\Elpath
	\caption{\label{fig:path} 
		Two choices after eliminating a $(3,1)$-path of length $\ell=3$ respecting the edge count; $\ell-1$ hanging edges get eliminated.
	}
\end{figure}

The second type eliminates a path of $\ell\ge 2$ consecutive vertices, all of extended degree $(3,1)$ (Figure~\ref{fig:path});
we avoid single $(3,1)$ vertex elimination because that would yield a looser bound.
Edges are eliminated similarly as before, namely we eliminate the $\ell$ hanging edges (one per vertex) and another $\ell$ normal edges incident to path vertices, thus eliminating $2 \ell$ edges.
After eliminating the path, there are two choices for the normal edge that remains; in either case, it becomes hanging.

Now, we introduce two parameters for controlling the elimination process, namely the cost and the hanging edge equilibrium.

In every elimination step, there are several ways to choose the edges that remain in the new pseudograph.
The number of choices corresponds to different pseudographs with valid orientations; their number is defined to be the \textit{cost} of the step.

\begin{remark}\label{Rcost}
	The product of the costs of all steps in the elimination process bounds the number of valid orientations of the initial pseudograph.
	In other words, the cost expresses the quotient of the valid orientations of the original graph over the maximum number of valid orientations of the resulting graphs.
\end{remark}

In the proposition that follows, we show that, for vertex elimination, the cost is determined by the extended degree of the eliminated vertex, while for path elimination, the cost always equals~2.

Another important quantity in the elimination is the \textit{hanging edge equilibrium}, defined as the difference between hanging edges in the resulting pseudograph and the original one.

\begin{proposition}\label{prop:eq}
	Let $v$ be a vertex with extended degree $(p,h)$, then the cost and the hanging edge equilibrium of the elimination step are given by
	\begin{align*}
		\binom{p-h}{2-h},  \text{ and  } & \text{    }  p-h-2
	\end{align*}
	respectively.
	In the case of path elimination, for a path of length $\ell$, the cost is $2$ and the hanging edge equilibrium is $1-\ell$.
\end{proposition}

\begin{proof}
	Recall that at vertex elimination, two edges are eliminated and, when there are hanging edges, these are eliminated first.
	So $2-h$ edges are left to be eliminated among the $p-h$ normal edges of the vertex, which yields the cost of this step.
	Since $2-h$ edges were eliminated, the number of the new hanging edges is $p-h-(2-h)$, while the initial number of hanging edges was $h$.
	Their difference yields $p-h-2$.
	
	Let us view path elimination as a sequence of vertex eliminations.
	Then, eliminating the first vertex has cost~2.
	Each following vertex now has degree $(2,1)$ or $(3,2)$, hence its elimination cost is~$1$.
	Therefore the overall cost is $2$ because it equals the product of all costs.
	As for the hanging edge equilibrium, the path contains $\ell$ hanging edges and, after the elimination step, one remains. 
\end{proof}

If the iterative process continued up to the exhaustion of vertices and, moreover, all cases were as in Figure~\ref{fig:el}(a), there would be $O(3^{|V|})$ orientations which, by Theorem~\ref{thm:bez}, yields a bound of $O(6^{|V|})$ on embeddings. 
However, our process is defined to terminate earlier; see Proposition~\ref{prop:tree}.

\subsection{Bounding the number of valid orientations.}\label{sec:pbound}

In this subsection, by applying the process described above, we bound the number of valid orientations of connected pseudographs.
In the sequel, $n$ denotes the number of vertices of a connected pseudograph and $k$ the total number of its hanging edges.

We first prove that there is always an elimination process that keeps the pseudograph connected.
For this, we recall the definition of a block-cut tree \cite[Chapter~4]{Harary69}.
Recall that a cut-vertex is such that its removal increases the number of connected components in the graph and a biconnected component is a maximal subgraph with no cut vertices \footnote{In \cite[Ch.~3]{Harary69} these subgraphs are called blocks; "biconnected component" is used equivalently, e.g.\ \cite[Ch.~8]{Jungnickel}.}.

\begin{defn}[Harary \cite{Harary69}]\label{def:bcg}
	Let $G(V,E)$ be any graph. Let also $bc(G)$ be the graph such that:
	\begin{itemize}
		\item This graph has a vertex for each biconnected component, and for each cut-vertex of $G$.
		\item There is an edge in $bc(G)$ for each pair of a biconnected component in $G$ and a cut-vertex that belongs to that block.
	\end{itemize}
	If $G$ is connected, then $bc(G)$ is a tree and is called block-cut tree of $G$.
\end{defn}

Following Definitions~\ref{def:pseu} and~\ref{def:bcg}, block-cut trees can be used in the case of normal subgraph $G(U,F)$ of a connected pseudograph $L(U,F,H)$ (Figure~\ref{fig:bcG}).

\begin{figure}[htp]
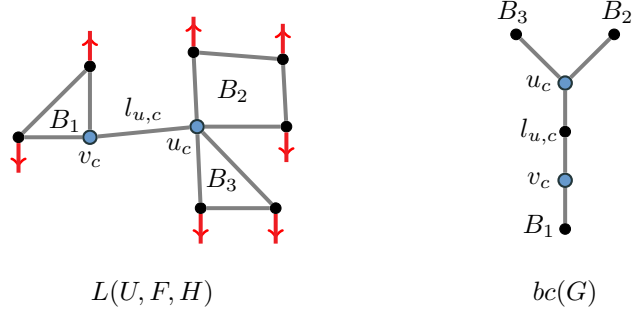

	\centering
	\begin{tabular}{cc}
		\bcG  \hspace{9mm} & \hspace{9mm} \Gbc	\vspace*{3mm}\\
		$L(U,F,H)$ \hspace{9mm} & \hspace{9mm} $bc(G)$
	\end{tabular}
	\caption{\label{fig:bcG} A pseudograph $L(U,F,H)$ and the block-cut tree of its normal subgraph $G(U,F)$.
	}
\end{figure}

We can now prove the following statement, which allows us to use the bound in Expression~(\ref{eq:p}).

\begin{proposition}\label{prop:dec}
	Given a connected pseudograph $L(U,F,H)$, there is always an elimination process where, in each step, we either eliminate a vertex with extended degree other than $(3,1)$, or we eliminate a $(3,1)$-path with length at least~2, so that the resulting pseudograph remains connected.
\end{proposition}

\begin{proof}
	If there is a non-cut vertex with degree other than $(3,1)$, then it can be eliminated and 	the proposition holds.
	
	We now show that, if all vertices in $L$ with degree other than $(3,1)$ are cut-vertices, then there are at least two adjacent $(3,1)$-vertices that can be eliminated keeping the pseudograph connected (an example is shown in Figure~\ref{fig:bcG}).
	Since $L$ is connected, its normal subgraph $G$ is connected as well, and there exist non-cut vertices in $G$; their extended degree must equal $(3,1)$.
	
	From the definition of block-cut trees, the leaves of $bc(G)$ represent biconnected components in $G$.
	In these components, all vertices but one are non-cut vertices, and their normal degree is~2, since their extended degree is $(3,1)$.
	If such biconnected component had only one non-cut vertex, then this vertex would have normal degree $1$.
	This means that there are at least two such vertices in a biconnected component of $L$ and, since their normal degree equals~2, there exists a path containing $\ell\ge 2$ such vertices, denoted $(v_1,\dots ,v_{\ell})$.
	This path can be eliminated and the resulting pseudograph remains connected; 
	more precisely, we may eliminate successively each $v_i$, thus making vertex $v_{i+1}$ have normal degree~1.
	
	This completes the proof. 
\end{proof}

Concerning the termination condition of our process, we establish the following for a connected pseudograph whose normal subgraph is a tree. 

\begin{proposition}\label{prop:tree}
	Let $L(U,F,H)$ be a connected pseudograph such that $G(U,F)$ is a tree.
	Then
	\begin{enumerate}
		\item The number of valid orientations for $L$ is either $1$ or $0$;
		\item If $L$ has a valid orientation, then $k=n+1$;
	\end{enumerate}
	where $n=|U|$ and $k=|H|$.
\end{proposition}

\begin{proof}
	Since $G(U,F)$ is a tree, we can always eliminate a vertex of normal degree $1$ from our pseudograph.
	This means that $p-h=1$, so using the formulas of Proposition~\ref{prop:eq} it is obvious that, if there is a valid orientation, then this is unique.
	If it does have a valid orientation, then from the total edge count in Proposition~\ref{prop:pseucount}, we deduce $|F|+k=2\cdot n$. 
	Since $G$ is a tree, we substitute $|F|=n-1$ in this formula, concluding the proof. 
\end{proof}

Let $P(n,k)$ denote the maximal number of valid orientations for all connected pseudographs with $n$ vertices and $k$ hanging edges. 
Let us recall the Br\`egman-Minc bound and the connection between permanents, constrained orientations, and the bound of Laman graphs as discussed in Section~\ref{sec:mbez} and in~\cite{bes}, where it was established that:
\begin{equation}
	\label{eq:BMp}
	P(n,k)  \le (2!)^{k/2} \, (4!)^{(2n-k)/4} \cdot (2)^{-n} 
	\approx 2.4495^n \cdot 0.6389^k.  
\end{equation}
We therefore seek upper bound estimates of the form
\begin{equation} \label{eq:p}
	P(n,k)\le \alpha^n\,\beta^k
\end{equation}
for real $\alpha,\beta>0$ and $k,n\ge 1$.

Proposition~\ref{prop:tree} implies $P(n,n+1)=1$ for every $n\geq 1$; this is the base case in Theorem~\ref{thm:fixed2}.
Proposition~\ref{prop:dec} precludes that multiple connected components be formed, thus leading to the theorem's inductive proof.
Additionally, Propositions~\ref{prop:pseucount} and~\ref{prop:tree} establish that $k\leq n+1$ for any connected pseudograph with at least one valid orientation. Indeed, $k>n+1$ implies the normal subgraph has $<n-1$ edges so cannot be connected.

We modify the form of the bound in Inequality~(\ref{eq:p}) to $P(n,k)\le\alpha^n\beta^{k-1}$, with $\alpha \beta>1$. 
The modification is justified in the proof below.

\begin{theorem}\label{thm:fixed2}
	The number of valid orientations for a connected pseudograph is bounded above by
	$$ 
	P(n,k) \leq \alpha^n \cdot \beta^{k-1},
	$$
	where $\alpha=24^{1/5}$ and $\beta = 18^{-1/5}$.
\end{theorem}

\begin{proof}
	We prove the statement by induction on $n,k$.
	The statement is true for the base cases $n=1$, $k=2$, which a pseudograph consisting of exactly one $(2,2)$ vertex, and also for trees with $k=n+1$, since $P(n,n+1)=1$ (Proposition~\ref{prop:tree}), because $\alpha\beta>1$.
	In these cases the pseudograph has 1 or 0 orientations, representing a termination condition.
	If the exponent of $\beta$ were $k$, the statement would fail for small trees.
	
	From Propositions~\ref{prop:pseucount} and~\ref{prop:tree}, if a connected pseudograph has $k>n$ hanging edges, either it is a tree, or it has no valid orientations.
	So we assume pseudograph $L$, with $n>1$ vertices, has $k \leq n$ hanging edges. 
	Suppose it has a vertex $v$ of extended degree $(p,h)$, such that:
	\begin{enumerate}
		\item $(p,h)\neq (3,1)$, and 
		\item elimination of $v$ and its incident edges keeps the pseudograph connected.
	\end{enumerate}
	Since the number of valid orientations of $L$ is bounded by the cost of the elimination process (Remark~\ref{Rcost}) and the hanging edge equilibrium is $p-h-2$, the number of valid orientations after eliminating this vertex is bounded by
	$$
	\binom{p-h}{2-h}\,P(n-1,k+p-h-2).
	$$
	By the induction assumption, this is bounded by $C(p,h)\,\alpha^n\beta^{k-1},$ where
	$$
	C(p,h)={p-h \choose 2-h} \,\alpha^{-1}\beta^{\,p-h-2}.
	$$
	We now prove that $C(p,h) \le 1$, for $p\ge 2\ge h\ge 0$, and $(p,h)\neq (3,1)$. 
	Direct substitution gives
	$$
	C(p,h)={p-h \choose 2-h} \, \left( 2^{h-p-1}\;3^{2h-2p+3} \right)^{1/5}.
	$$
	Note that:
	\begin{itemize}
		\item $C(2,0)=24^{-1/5}<1$, $C(3,0)=(9/16)^{1/5}<1$, $C(4,0)=1$,
		and the $C(p,0)$ for $p>4$ are decreasing with $p$ as follows:
		\begin{equation}
			\frac{C(p+1,0)}{C(p,0)} = \left( 1+\frac{2}{p-1} \right)
			\beta <1, \quad	\mbox{for } p\ge 4.
		\end{equation}
		\item $C(2,1)=(3/4)^{1/5}<1$, $C(4,1)=(9/16)^{1/5}<1$, 
		and the $C(p,1)$ for $p>4$ are decreasing with $p$ as follows:
		\begin{equation}
			\frac{C(p+1,1)}{C(p,1)} = \left( 1+\frac{1}{p-1} \right) 
			\beta <1, \quad
			\mbox{for } p\ge 4.
		\end{equation}
		\item $C(3,2)=(3/4)^{1/5}<1$, and the $C(p,2)$ for $p>3$ are strictly decreasing with $p$,
		as the binomial factor equals the constant 1.
	\end{itemize}
	$C(2,2)$ is immaterial since $(2,2)$ is a base case corresponding to a pseudograph with a single vertex and $k>n$.
	
	An induction step is proven under the assumptions \refpart{i}--\refpart{ii}. 
	Incidentally, $C(3,1)=(4/3)^{1/5}>1$, which is why we avoid eliminating this type of vertices in a vertex elimination step.
	
	If assumptions \refpart{i}--\refpart{ii} fail, we can eliminate a path of $(3,1)$-vertices keeping the pseudograph connected by Proposition~\ref{prop:dec}.
	
	Let $\ell \ge 2$ denote the number of vertices in the eliminated path.  
	Then, the number of orientations of $L$ is bounded by
	$$
	2P(n-\ell,k-\ell+1),
	$$
	which, by induction, is bounded by
	$$
	2\,\alpha^{n-\ell}\beta^{k-\ell}
	= \left( \frac34 \right)^{\!(\ell-2)/5}\alpha^n\beta^{k-1}
	\, \le \; \alpha^n\beta^{k-1}.
	$$
	The bound is proven. 
\end{proof}

\subsection{A new upper bound on the embedding number of Laman graphs}\label{sec:blam}

This subsection combines the above discussion so as to establish a new upper bound on the  number of embeddings for Laman graphs.

Let $G(V,E)$ be a Laman graph and a fixed edge $e=(v_1,v_2) \in E$.
Let also $L_{G,e}(U,F,H)$ be a collection such that $U=V \backslash\{v_1,v_2\}$, $F=\{e'\in E : v_1, v_2 \notin e'  \}$ and $H$ is the set of all edges incident to one fixed vertex and one non fixed-vertex.
Then $L_{G,e}$ is a pseudographs that may contain one or multiple connected components; 
in Figure~\ref{fig:DisCon}, this construction leads to a pseudograph with two connected components. 
Remark that the number of vertices $n$ of $L_{G,e}$ is related to the number of vertices of $G$ by $n=|V|-2$.

\begin{figure}[htp]
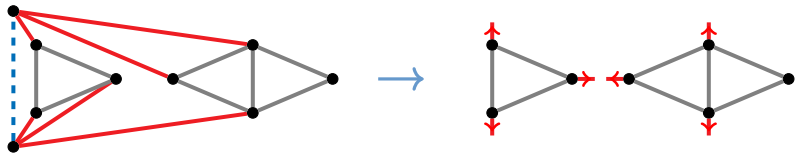

	\centering	\DisCon
	\caption{\label{fig:DisCon} After removing a fixed edge (vertical dashed blue) from a Laman graph, one gets a pseudograph with $2$ connected components.
	}
\end{figure}

A different choice of a fixed edge may result in different pseudographs, for a given Laman graph, while different Laman graphs may result in the same pseudograph, see Figure~\ref{fig:LDiff}.
This happens because any pseudograph representation lacks the information on connections with specific vertices of the fixed edge.

\begin{figure}[htp]
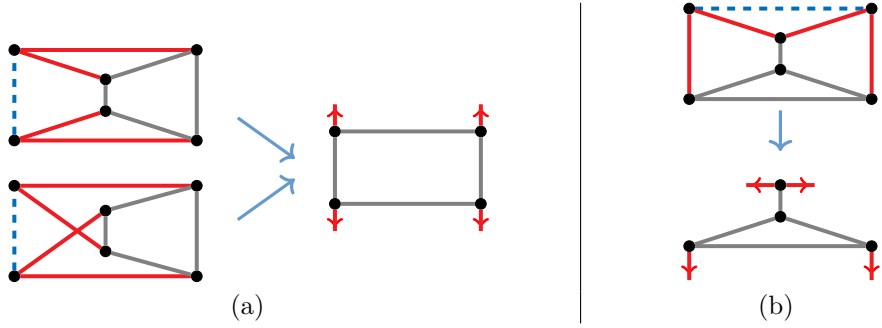

	\centering \begin{tabular}{c|c}
		\LDiff  \hspace{9mm} & \hspace{9mm} \Destr	\\
		(a) \hspace{9mm} & \hspace{9mm} (b)
	\end{tabular}
	\caption{\label{fig:LDiff} (a) Two Laman graphs, Desargues' and $K_{3,3}$, both resulting in the same pseudograph for some fixed edge. (b) Choosing a different fixed edge for Desargues' graph results in a different pseudograph.
	}
\end{figure}

From the construction of $L_{G,e}$ it follows that, when it is connected, its number of valid orientations equals that of its constrained orientations defined in Theorem~\ref{thm:bez}.
This bound is always positive, since it corresponds to well-constrained algebraic systems.
If $L_{G,e}$ has $c>1$ connected components $L_1, \dots, L_c$, then the total number of valid orientations of $L_{G,e}$ equals the product of valid orientations per connected component $L_i$.
This leads to the following corollary, which distinguishes components with one vertex in order to exploit Maxwell's count.

\begin{cor}\label{cor:Lam}
	Let $G(V,E)$ be a Laman graph and $L_{G,e}$ constructed as described above.
	Let $c'$ be the number of connected components of $L_{G,e}$ with more than one vertex, and $n$ the number of its vertices.
	Then, the number of constrained orientations, defined in Theorem~\ref{thm:bez}, is bounded above by 
	$$
	24^{n/5}\cdot 18^{-(k'-c')/5},
	$$
	where $k'$ is the total number of hanging edges in the components of $L_{G,e}$ with more than one vertex.
\end{cor}
\begin{proof}
	Recall that $n=|V|-2$.
	Let $n_1, n_2 \dots n_{c'}$ and $k_1, k_2 \dots k_{c'}$ be respectively the numbers of vertices and of hanging edges per connected component with strictly more than one vertex.
	The bound follows from Theorem~\ref{thm:fixed2}, since $n\geq  \sum\limits_{i=1}^{c'} n_i$ and $k'= \sum\limits_{i=1}^{c'} k_i$.	
\end{proof}

\begin{lemma}\label{lem:kc}
	Let $G(V,E)$ be a Laman graph, and $L_{G,e}$, $k'$ and $c'$ as above. Then $k' \geq 3c'$.
\end{lemma}

\begin{proof}
	Let $L_i(U_i,F_i,H_i)$ be a connected component of $L_{G,e}$ with $k_i$ hanging edges.
	The normal subgraph $G_i(U_i,F_i)$ is a subgraph of a Laman graph.
	If $|U_i|\geq 2$, by Maxwell's count we have $|F_i|\leq 2 \cdot |U_i|-3$ therefore $k_i\geq 3$. 
\end{proof}

Now we are ready to prove the new upper bound for Laman graphs.

\begin{theorem}\label{thm:aslam}
	Let $G(V,E)$ be a Laman graph.
	Then the number of its embeddings in $\CC^2$ (and $S^2$) is bounded from above by 
	$$
	18^{-2/5}\cdot \left(4\cdot (3/4)^{1/5}\right)^{|V|-2} = O\left(3.7764^{|V|}\right).
	$$
\end{theorem}

\begin{proof}
	Applying $k' \geq 3c'$ from Lemma~\ref{lem:kc} in Corollary~\ref{cor:Lam}, the number of valid orientations is bounded by $24^{n/5} \cdot 18^{-2/5}$, for $n\geq2$, since either the number of connected components with more than one orientation is $c'\geq 1$, or there is a single valid orientation.
	By doubling this bound and applying Theorem~\ref{thm:bez}, the upper bound follows.
	For $n=1$, Lemma~\ref{lem:kc} does not apply; there is trivially one orientation and the bound is $2$.
\end{proof}

\section{Geiringer graphs and higher dimensions} \label{sec:higher}

This section extends the method of the previous section to orientations of connected pseudographs with fixed outdegree $d\geq 3$, and subsequently establishes new upper bounds on the embedding number of minimally rigid graphs in $\CC^d$ (and $S^d$), for $d\ge 3$.

Let $P_d(n,k)$ denote the maximal number of orientations with fixed outdegree $d$ for connected pseudographs with $n$ vertices and $k$ hanging edges.
As before, we seek bounds of the form
$$
P_d(n,k)\leq \alpha_d^n \cdot \beta_d^{k-1}
$$
for each $d$.
For a fixed outdegree $d\ge 3$, the elimination steps consist of:
\begin{itemize}
	\item Eliminating single vertices of extended degree $(p,h)$, with $p\geq d\geq h\geq 0$, and $(p,h)\neq (d+1,d-1)$;
	then the number of valid orientations is bounded by $ \displaystyle\binom{p-h}{d-h} \cdot P_d(n-1,k+p-h-d).$
	\item Eliminating paths of length $\ell \ge 2$ with $(d+1,d-1)$-vertices;
	then the number of valid orientations is bounded by  
	$ 2 \cdot  P_d(n-\ell,k-(d-1) \ell+1)$. 
\end{itemize}

If we replace $(3,1)$-paths in Proposition~\ref{prop:dec} by $(d+1,d-1)$-paths, we have an analogous result, since $(d+1,d-1)$-vertices have normal degree~$2$.
This implies that there is always an elimination process preserving connectivity.
Moreover, the necessary count in Proposition~\ref{prop:pseucount} is generalized to $|F|+|H|=d\cdot |U|$ for every pseudograph with at least one orientation with fixed outdegree $d$; such orientations extend the notion of validity beyond $d=2$.

An immediate consequence is that, if a connected pseudograph has a tree as normal subgraph and also has an orientation with fixed outdegree $d$, then it holds that $(d-1) n= k-1$ which is our base case, generalizing Proposition~\ref{prop:tree}.

\begin{figure}[htp]
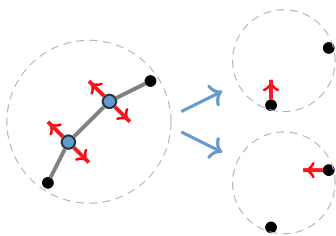

	\centering \Elpathd
	\caption{\label{fig:pathd} Elimination of a $(4,2)$-path (with length $\ell=2$) in the case of orientations with fixed outdegree $3$.
		This elimination is analogous to that in Figure~\ref{fig:path}.
	}
\end{figure}

In the following theorem we establish an upper bound on $P_d$.
If $d=2$, then $\alpha_2, \beta_2$ are evaluated as in Theorem~\ref{thm:fixed2}.
Here, elimination of single vertices of extended degree $(d,d)$ and $(d+1,d-1)$ are excluded from our analysis. The first case because it is one of the base cases, as a pseudograph with exactly one vertex,
and the latter because only path elimination with length $\ell \geq 2$ is considered for these vertices.

\begin{theorem}\label{thm:gen}
	The maximal number of orientations with fixed outdegree $d$ for a connected pseudograph is bounded from above by
	$$P_d(n,k) \leq \alpha_d^n \cdot \beta_d^{k-1}  $$
	for the following choices of $\alpha_d$ and $\beta_d$:
	\begin{equation} \label{eq:ad}
		\alpha_d = \max_{p\ge d} \left(2^{p-d} \, {p\choose d}^{\!2d-3} \right)^{\displaystyle\frac{1}{2p-3}},
	\end{equation}
	and 
	\begin{equation} \label{eq:bd}
		\beta_d=\left(2\,{p\choose d}^{\!-2}\right)^{\displaystyle\frac{1}{2p-3}}
	\end{equation}
	for the value of $p$ that maximizes $\alpha_d$.
\end{theorem}

\begin{proof}
	The single vertex elimination step and the path elimination step result in the following inequalities:
	\begin{equation} \label{eq:genineqa}
		\alpha_d\,\beta_d^{d+h-p} \ge {p-h\choose d-h}
	\end{equation}
	for all $(p,h)\not\in\{(d+1,d-1),(d,d)\}$ with $p\ge d\ge h$, and 
	\begin{equation} \label{eq:genineqb}
		\alpha_d^\ell\,\beta_d^{\,d\ell-\ell-1} \ge 2
	\end{equation}
	for all $\ell\ge 2$. 
	In the second case the equalities are achieved with $\alpha_d=2^{d-1}$, $\beta_d=1/2$ for all $\ell$.
	The same $(\alpha_d,\beta_d)$-point gives equality in (\ref{eq:genineqa}) for $(p,d)\in\{(d,d-1),(d+1,d)\}$. 
	
	By taking the logarithm, (\ref{eq:genineqa}) and~(\ref{eq:genineqb}) become linear in the $(\ln\alpha_d,\ln\beta_d)$-plane. 
	The corresponding lines have negative slope and contain point $((d-1)\ln2,-\ln 2)$; the one defined by (\ref{eq:genineqb}) for $\ell=2$ is closest to the vertical.
	So the corresponding inequalities are dominated for \mbox{$\alpha_d\le 2^{d-1}$} by (\ref{eq:genineqb}) with $\ell=2$. 
	
	Our key observation is that (\ref{eq:genineqa}) for a relevant pair $(p,h)$ is satisfied if:
	\begin{enumerate}
		\item The same inequality is satisfied for the shifted $p\to p+1$, $h\to h-1$, giving
		$$
		\alpha_d\,\beta_d^{d+h-p-2} \ge {p-h+2\choose d-h+2} .
		$$
		\item Inequality~(\ref{eq:genineqb}) is satisfied for $\ell=2$.
	\end{enumerate}
	This implies that it is enough to consider~(\ref{eq:genineqa}) with $h=0$. 
	Considering \refpart{ii} with $\ell=2$ and a particular case of (\ref{eq:genineqa}) with $h=0$,
	the two inequalities can be raised to non-negative powers and combined so as to eliminate $\beta_d$, with the conclusion that
	\[
	\alpha_d^{2p-3}\ge 2^{p-d} \, {p\choose d}^{\!2d-3}.
	\]
	A permissible equality is achieved together with Equality~(\ref{eq:bd}). 
	The maximization in Equality~(\ref{eq:ad}) through $p\ge d$ follows.
	
	It remains to prove the key observation. 
	Inequalities \refpart{i}--\refpart{ii} can be raised to positive powers
	and combined, with the conclusion that
	\[
	\alpha_d\,\beta_d^{d+h-p} \ge {p-h+2\choose d-h+1}^{\frac{2p-2h-3}{2p-2h+1}} \, 2^{\frac{2}{2p-2h+1}}.
	\]
	Positivity fails for $(p,h)\in\{(d,d),(d,d-1),(d+1,d)\}$, but these cases are covered already.
	
	Comparison with Inequality~(\ref{eq:genineqa}) shows that we need $U(p-h,d-h)\le 1$, where
	\[
	U(x,y)=\frac{1}{4} \,{\,x\, \choose y}^{\!4} \,\left(\frac{(x-y+1)(y+1)}{(x+1)(x+2)} \right)^{\! 2x-3}.
	\]	
	The inequality has to be shown for integer $x\ge y\ge 0$ such that $(x,y) \notin \{(2,1),(0,0),(1,0),(1,1)\}$, which correspond respectively to vertices $(d+1,d-1), (d,d), (d+1,d)$, $(d,d-1)$.
	For fixed $y$, the maximum is achieved at $x=2y$, since $U(x,y)$ increases. 
	This follows from
	\begin{align*}
		\frac{U(x+1,y)}{U(x,y)} 
		= \frac{(y+1)^2}{(x-y+1)(x-y+2)}\left(1-\frac{1}{(x+2)^2}\right)\left(1-\frac{x-2y+1}{(x+3)(x-y+1)}\right)^{\!2x}. 
	\end{align*}
	This ratio is $< 1$ for $x\ge 2y$, and (by calculus on the product of the first and middle terms) it is $> 1$ for $x<2y$.
	
	For this pair of values, we have
	\[
	U(2y,y)=2 \prod_{k=1}^y \left(1-\frac{1}{k\,(2k+1)}\right)^{\!4k-3}.
	\]
	Evidently, the $x$-maximum $U(2y,y)$ is a decreasing function of $y$, and $U(4,2)=3^9/25000<1$. 
	For $y<2$, we observe that $U(x,1)\le U(3,1)=3^7/4000<1$ and $U(x,0)<U(1,0)=3/4<1$ for $x>2$.
\end{proof}

We now describe a construction similar to Section~\ref{sec:blam}, connecting Theorem~\ref{thm:bez} to the orientations of pseudographs with fixed outdegree $d$ and $n>1$.
Let $G(V,E)$ be a minimally rigid graph in $\CC^d$ and let $K_d$ be one of its subgraphs.
Removing $K_d$, as in the case of $d=2$, we have a pseudograph, denoted $L_{G,K_d}(U,F,H)$.

Applying Maxwell's condition to the edge count for the normal subgraphs of the pseudographs, we obtain $k_i\geq  \binom{d+1}{2}$ for every connected component of $L_{G,K_d}$ with at least $d$ vertices.

If $n_i\leq d-1$ we examine 2 different cases:  
(a) $n_i=1$ or $2$ implying that the connected component has trivially one orientation, since the normal subgraph is either a single vertex, or a tree with $2$ vertices.
The number of hanging edges is $k_i=2d-1$, in order to respect the total edge count.
(b) $n_i=3$, which implies $d\geq 4$, then there should be $\le 3$ normal edges and $\ge d-2$ hanging edges, so $k_i\geq 3d-3$.
By induction, for the $i$-th vertex, at least $i-1$ hanging edges shall be added.

Hence, if $L_{G,K_d}$ has $c'$ connected components with more than 2 vertices and a total of $k'$ hanging edges, then $k'\geq 3d-3$ for $d\geq 3$. 

An immediate consequence is the following theorem that generalizes Theorem~\ref{thm:aslam} by applying Theorem~\ref{thm:bez}.

\begin{theorem}
Let $G(V,E)$ be a minimally rigid graph in $\CC^d$ (and $S^d$) that contains a $K_d$.
Then the number of its embeddings is bounded from above by $\beta_d^{3d-2} \cdot(2\cdot a_d)^{|V|-d}= O((2\cdot a_d)^{|V|})$.
In the case of Geiringer graphs that contain a triangle this is 
$$
\left(2\cdot 10^2\right)^{-5/9}\left(8\cdot (5/8)^{1/3}\right)^{|V|-3} = O\left(6.8399^{|V|}\right).
$$
\end{theorem}

The asymptotic bound works also for minimally rigid graphs in $\CC^d$ that do not contain a $K_d$.
In that case, we may remove a maximal clique with $d'\geq 2$ vertices in order to obtain a pseudograph.
Then, the exponent of $a_d$ will never exceed $|V|-2$ (see Section~\ref{sec:mbez} for details).\\

Let us discuss succinctly the asymptotics of $\alpha_d$ in Equality~(\ref{eq:ad}).
We rely on the following consequence of Stirling's formula:
\begin{equation} \label{eq:cbinstir}
{2d+z\choose d} \sim  \frac{2^{2d+z}}{\sqrt{\pi d}} \left(  1+O\!\left(\frac{z^2}{d}\right)\right) ,
\end{equation}
for $z=\Omega(1)$ and $z=o(\sqrt{d})$ as $d\to\infty$.  
For $z=o(1)$, the error term would be $O(1/d)$.
We shall show that the maximum in Equality~(\ref{eq:ad}) is achieved at
\begin{equation} \label{eq:optloc}
p=\textstyle 2d+\frac12\ln d+O(1), 
\end{equation}
giving the following asymptotic expression for $\alpha_d$:
\begin{equation}\label{eq:asym}
\alpha_d\sim \sqrt{\frac12\,{2d\choose d}} \, \left(1+O\!
\left(\frac{(\ln d)^2}{d}\right)\right).
\end{equation}
The expression for $\alpha_d$ is close enough to the intuitive location $(p,h)=(2d,0)$ of the optimum, suggesting the same leading asymptotic term:
\begin{align*}
\left( 2^{d}\;{2d \choose d}^{\!-3/2} \right)^{\!1/(4d-3)} \! \sqrt{2d \choose d}
& \sim \left(\frac{\pi d}{4}\right)^{\!3/(16d-12)} \sqrt{\frac12\,{2d \choose d}}
\, \left( 1+O\!\left(\frac{1}{d}\right)\right)
\nonumber \\
& \sim  \sqrt{\frac12\,{2d \choose d}} \, \left( 1+O\!\left(\frac{\ln d}{d}\right)\right).
\end{align*} 

This is an improvement of Equation~(\ref{eq:bes}) from~\cite{bes} by the asymptotic factor of $1/\sqrt2$. 
To locate the maximum in Equality~(\ref{eq:ad}), let $\alpha_d(p)$ denote the function to be optimized.
We have
\[
\frac{\alpha_d(p)}{\alpha_d(p+1)}=W(p)^{\frac{2d-3}{(2p-1)(2p-3)}}
\]
with
\begin{align}
W(p) & =\frac12\,{\,p\, \choose d}^{\!2} \,\left(\frac{p-d+1}{p+1} \right)^{\! 2p-3} \\
& = \frac1{2\,(d+1)^{2d-3}} \; \prod_{k=d+1}^p \left( 1+\frac{d}{(k-d)(k+1)} \right)^{2k-3}. \nonumber
\end{align}
Evidently, $W(p)$ is increasing with $p\ge d$ from a value $<1$.
The logarithm
\begin{equation*}
\ln W(p)= -\ln2+2\ln{p\choose d}+(2p-3)\ln \left( 1-\frac{d}{p+1} \right)
\end{equation*}
increases continuously from a negative value to $\infty$, 
hence it has exactly one zero indicating the maximum of $\alpha_d(p)$ up to $O(1)$.
We find the zero assuming $p=2d+o(\sqrt{d})$ and using (\ref{eq:cbinstir}) again.
Then, the right-hand side is asymptotically
\[
2\ln2-\ln(\pi d)+2(z+1)+O\!\left(\frac{z^2}{d}\right),
\; \mbox{ with } z=p-2d.
\]
The optimal location for $p$ follows; see Equation~(\ref{eq:optloc}).

To demonstrate the improvement achieved by our new bound on the embedding number of rigid graphs, namely $(2\alpha_d)^{|V|-d}$, we refer the reader to Table~\ref{tab:bounds} which compares the values of $2\alpha_d$ to the power basis of existing bounds.

\section{Conclusions and future work}\label{sec:conc}

In this paper we introduce a method that bounds the number of outdegree constrained eliminations that are related to a recently proposed bound for the embeddings of minimally rigid graphs.
This method resulted in a new bound for the embeddings of all minimally rigid graphs with a given number of vertices, which was generalized as the first non-trivial upper bound in the cases of Laman and Geiringer graphs, while it also improved existing bounds in higher dimensions.

Our results give rise to certain open questions.
First of all, we would like to investigate how sharp our upper bound is on the number of pseudograph orientations and, subsequently, on the maximal number of embeddings of minimally rigid graphs.
Both of these may require large computational resources.

\paragraph{Acknowledgements}
	EB was fully supported and IZE was partially supported by project ARCADES which has received funding from the European Union’s Horizon 2020 research and innovation programme under the Marie Sk\l{}odowska-Curie grant agreement No~675789. 
	EB and IZE are members of team AROMATH, joint between INRIA Sophia-Antipolis, France, and NKUA.

\bibliography{upper_laman}

\end{document}